\documentclass[11pt,reqno]{amsart}

\usepackage{amsfonts, amssymb, amsthm, amsmath, bbm, wasysym, enumerate, amsxtra, latexsym, fullpage, amscd, graphics, epic, sansmath, mathtools, multicol, url, hyperref, bbm, blindtext, wasysym, comment, datetime, stmaryrd,mathrsfs}
\usepackage[shortlabels]{enumitem}
\usepackage[all,cmtip]{xy}
\usepackage[dvipsnames]{xcolor}
\usepackage[makeroom]{cancel}  
\usepackage[enableskew]{youngtab}
\usepackage{tikz}
\usetikzlibrary{cd}

\usepgflibrary{shapes.geometric}

\tikzstyle{V}=[draw, fill =black, circle, inner sep=0pt, minimum size=1.5pt]
\tikzstyle{C}=[draw, fill =white, circle, inner sep=0pt, minimum size=1.5pt]
\tikzstyle{over}=[draw=white,double=black,line width=2pt, double distance=.5pt]
\usetikzlibrary{arrows, positioning, calc, chains, cd}
\tikzset{
	ch/.style={circle,draw,on chain,inner sep=2pt},
	chj/.style={ch,join},
	every path/.style={shorten >=4pt,shorten <=4pt}
	}

\numberwithin{equation}{section}
\usepackage[margin = 1in]{geometry}
\theoremstyle{definition}

\newtheorem{theorem}{Theorem}[section]

\newtheorem{proposition}[theorem]{Proposition}
\newtheorem{corollary}[theorem]{Corollary}
\newtheorem{definition}[theorem]{Definition}
\newtheorem{remark}[theorem]{Remark}
\newtheorem{question}[theorem]{Question}
\newtheorem{conjecture}[theorem]{Conjecture}
\newtheorem{observation}[theorem]{Observation}
\newtheorem{example}[theorem]{Example}

\DeclareMathOperator{\Hom}{{Hom}}

\def\<{\langle}
\def\>{\rangle}

\newcommand{\CHAR}{\text{characteristic}\:}

\newcommand{\Ext}{\text{Ext}}

\newcommand{\calO}{\mathcal{O}}

\newcommand{\PP}{\mathbb{P}}
\newcommand{\PGL}{\text{PGL}}

\newcommand{\rem}{\text{rem}}
\newcommand{\Res}{\text{Res}}

\newcommand{\sing}{\text{sing}}
\newcommand{\ord}{\text{ord}}

\newcommand{\Char}{\text{char}}

\newcommand{\subOrd}{\textup{subOrd}}

\newcommand{\Z}{\mathbb{Z}}

\allowdisplaybreaks

\title{Examples relating to Green's conjecture in low characteristics and genera}

\address{Department of Mathematics, University of Georgia, Athens, GA 30602 (Estu, Im, Manning, Michaels, Rulla)}

\address{Department of Mathematics, Emmanuel College, Franklin Springs, GA 30639 (Pasko)}

\address{Department of Mathematics, Augusta State University, Augusta, GA 30904 (Wijesinghe)}

    \author[Topik Teguh Estu]{Topik Teguh Estu}
 
	\author[Mee Seong Im]{Mee Seong Im}
\address{Department of Mathematical Sciences, United States Military Academy, West Point, NY 10996 (Im: Point of Contact)}
    \email{meeseongim@gmail.com}

    \author[Benjamin Manning]{Benjamin Manning}
	\address{}
    \email{} 

        \author[Zachary Michaels]{Zachary Michaels}
\address{}
    \email{} 

    \author[Joseph Pasko]{Joseph Pasko}
 \address{}
    \email{}

    \author[William Rulla]{William Rulla}
 \address{}
    \email{}     
    
    \author[Nishan Wijesinghe]{Nishan Wijesinghe}
 \address{}
    \email{}

    
\begin{document}

\begin{abstract} 
We exhibit approximately fifty Betti diagrams of free resolutions of rings of smooth, connected canonical curves of genera $9$-$14$ in prime characteristics between $2$ and $11$.  Generic Green's
conjecture is verified for genera $9$ and $10$ for characteristics $2$, $5$, $7$, and $11$. 
\end{abstract}  
 
\maketitle 

\section{Introduction}

This paper is the product of a four week VIGRE REU participated by the authors\footnote{The addresses provided are the authors' affiliations at the time of when drafting this manuscript. The primary point of contact and her current affiliation are also provided at the end of this paper.} and conducted by the sixth author during the Summer of
2002, together with further participation supported by UGA's Summer
Undergraduate Research Program. The research collaboration was an examination of
Gr\"obner bases and their use in calculating free resolutions, with
applications to an experimental study of a conjecture of Mark Green (cf. \cite{Gre84}).
It was inspired by
\cite{Eis92,Eis02,CLO97,CLO98,Eis95}.

The conjecture concerns Riemann surfaces, but carries over to
algebraic curves.  Throughout this manuscript, 
the terms ``curve'' and
``algebraic curve'' will mean smooth, proper, non-hyperelliptic
algebraic curve.  
Each curve of genus $g$ is associated a sequence
of non-negative integers of length $\lfloor \tfrac{g-3}{2}\rfloor$;
the sequence can be thought of as a refinement of the invariant $g$.
The conjecture relates the number of zero entries of the sequence to
geometric characteristics of the curve. This paper is an empirical
examination of the kinds of sequences which can occur.  It is divided as follows.  

\subsection{Summary of the sections}
\label{subsection:summary}
In \S\ref{S:background}, we provide some background on free resolutions and canonical embeddings of curves, establish notation,
and present the motivating question of our investigation. 
In \S\ref{S:conjecture}, we present Green's conjecture and generic Green's
conjecture. 
We describe in  
\S\ref{S:REU} a synopsis of our summer research.  
Algorithms used to examine the nonsingularity and connectedness of our examples are described in \S\S\ref{S:singular} and \ref{S:connected}.
An overview on how we generated our examples, what we expected to find, and our results is given in \S\ref{S:computations}.   
We give 
the verification of generic Green's
conjecture in several small characteristics and genera $g \le 10$ in \S\ref{S:genericGreen}, and we present 
examples of Betti diagrams corresponding
to singular and reducible curves in \S\ref{S:reducible}.  
In \S\ref{S:further}, we pose questions and give suggestions for further
projects. 

\subsection*{Acknowledgment}
We thank Maurice D. Hendon and D. Wayne Tarrant for their assistance during our preparation of this manuscript, 
and to Clint McCrory for providing us this collaboration opportunity. We also thank Robert Varley for extensive encouragement and helpful feedback throughout the duration of our research, and during the preparation of this paper. All authors also thank William (Bill) Rulla for teaching the research group computational algebraic geometry and homology algebra when we were undergraduate students. 
He spoke about 120 words per minute, but nonetheless, he is an exceptional and admirable researcher. 
All authors, with the exception of M.S.I., were partially supported by University of Georgia's NSF VIGRE grant DMS \#0089927. M.S.I was partially supported by SURP, sponsored by the University of Georgia Graduate School.   

\section{Background on Canonical Curves and Free Resolutions}
\label{S:background} 

The following is a synopsis of background found in \cite{ACGH85,Eis92,Eis95,Sch86}. Let
$k$ be any algebraically closed field.

Any (nonhyperelliptic) curve $C$ has an associated canonical  embedding
in the projectivization $\PP(H^0(C,K_C)) \cong \PP^{g-1}$ of its
$g$-dimensional $k$-vector space of global regular differential forms,
obtained by identifying a choice of homogeneous coordinates of
$\PP^{g-1}$ with a basis.  The image of the embedding corresponds to a
homogeneous ideal $I = I_C$ of the coordinate ring $R :=
k[X_0,\ldots,X_{g-1}]$ of $\PP^{g-1}$. The ring $R$ is graded by degree:
\[
	R = \bigoplus_{d\ge 0} R_d,
\]
where $R_d$ is the $k$-vector space of monomials of degree $d$ in
$X_0$, \ldots, $X_{g-1}$.  The ideal $I$ correspondingly decomposes into a direct sum 
\[
	I = \bigoplus_{d\ge 0} I_d 
\]
of vector spaces, 
where $I_d$ is the subset of homogeneous polynomials of degree $d$ in
$R$ vanishing identically on $C$.  The dimensions of the $I_d$'s are
determined by the common Hilbert function for all canonical curves of
genus $g$, so they provide no distinguishing information between curves of
the same genus.

More precisely, the Hilbert polynomial  
\[
	p(t) = (2t-1)(g-1) 
\] 
of such a curve gives us the dimensions of the $I_d$'s.  
This polynomial gives $h^0(C,K_C^{\otimes d})$ for all integers $d \ge
2$, while $h^0(C,K_C) = g$.  
For each degree, there is a sub-exact sequence
\[
	0 \to I_d \to R_d \to (R/I)_d \to 0,
\]
so $\dim I_d = \dim R_d - \dim (R/I)_d$.  Furthermore, 
\[
	\dim R_d = \binom{d+g-1}{g-1}.  
\]
We see that for fixed $g$, the dimension of $I_d$ is
independent of $C$.  
The graded ring $R$ can be realized as follows:
\[
	R = \bigoplus_{d \ge 0} H^0(C,{K_C}^{\otimes d}),
\]
where each summand is just the space of homogeneous polynomials in the
coordinates of $\PP^{g-1}$.  Bases of the $H^0(C,{K_C}^{\otimes d})$
were determined by Petri (cf. \cite[pages 127--130]{ACGH85}).
Hence we see that bases depend only on the genus of $C$.

More interesting is a theorem of Petri \cite[page 131]{ACGH85}, which
states that any canonical ideal is generated by homogeneous elements
of degree two, unless $C$ is trigonal, i.e., it admits a 3:1 map to
$\PP^1$, or is isomorphic to a plane quintic. In these cases some
generators must be of degree three.  Thus any canonical curve is
the (scheme-theoretic) intersection of degree two and (possibly)
degree three hypersurfaces,
i.e., there is an exact sequence
\begin{equation}\label{E:surj} 
\begin{matrix}
R(-3)^{\oplus b_1} \\
\bigoplus \\
R(-2)^{\oplus a_1}
\end{matrix}
\xrightarrow{f} I \to 0.  
\end{equation}
Note that no differential form $\omega \neq 0$ vanishes identically on
$C$, so $I$ contains nothing of degree 1. 
Here, e.g., the ``$(-2)$'' indicates that each of the $a_1$ generators
of the free module $R^{\oplus a_1}$ is mapped to a homogeneous
polynomial of degree $2$ in $I \subseteq R$.  By Petri, the exponent
$b_1$ is nonzero if and only if $C$ is trigonal or isomorphic to a plane quintic.

The proof of Petri's theorem involves an examination of the syzygies
(relations) on the generators of $I$, i.e., a calculation of generators
of the kernel of the map $f$ in the sequence \eqref{E:surj}  (see \cite[pages 131--135]{ACGH85}).  
A natural extension of this idea is to then
examine the kernel of a surjection
\[
	R^{\oplus m} \to \ker f,
\]
where $m$ is the (minimal) number of generators of $\ker f$ over $R$.
Iterating leads to an examination of a free
resolution of $I$.  It is a fact that a (finite) minimal free
resolution
\[
0 \to R^{\oplus n_r} \to \ldots \to R^{\oplus n_2} \to 
R^{\oplus a_1 + b_1} \xrightarrow{f} I \to 0
\]
exists, and is unique up to isomorphism of exact sequences.  

Next, we discuss the structure of minimal resolutions of ideals of
canonical curves.  To simplify the discussion, given an ideal $I$ we
consider instead the deleted resolution of the canonical ring $R/I$,
\begin{equation}\label{E:seq}
	0 \to R^{\oplus n_r} \overset{\varphi_r}{\rightarrow} \ldots \to
		R^{\oplus n_2} \to R^{\oplus a_1 + b_1} \to R,
\end{equation}
where the cokernel of the rightmost map is $R/I$.  Rings of canonical
curves are Cohen--Macaulay, so the number $r$ determining the length of
the resolution is known by the formula of Auslander--Buchsbaum to be
$g-2$.

As outlined above, the second term of the sequence (\ref{E:seq})
decomposes as a sum of $R(-2)$'s and $R(-3)$'s; we can similarly be
more specific with the others.  Since the resolution is minimal, there
is never a ``degree zero'' relation on the elements of any kernel, so
the twisting must increase at each step.  Thus for example there are
no $R(-d)$ terms in the $R^{\oplus n_2}$ term of (\ref{E:seq}) for $d
= 0,1,2$.

Let $i:C\to \PP^{g-1}$ be the canonical morphism.  By considering what
happens when the contravariant functor
\[
\Hom_R(\cdot, \calO_{P^{g-1}}(-g-1))
\]
is applied to the sheafification of (\ref{E:seq}), using 
\[
\Ext^i_{P^{g-1}}(i_*\calO_C, \calO_{P^{g-1}}(-g-1)) \cong 
\begin{cases} 
0 & \text{if }0 < i < g-2 \\
i_*\calO_C & \mbox{if }i = g-2, \\ 
\end{cases}
\]
and the fact that $C$ is projectively normal, 
one finds that by reversing the arrows of (\ref{E:seq}) one gets another
deleted resolution of $R/I$.  
This ``dual'' sequence 
is therefore also a deleted resolution for $R/I$,
must be minimal, and by uniqueness is isomorphic to the
original (this is the ``Gorenstein'' property of $R/I$).  The symmetry
puts a serious restriction on the form of the resolution; in fact, any
such must be of the form
\begin{multline}\label{E:exactSeq} 
	0 \to R(-g-1) \to 
\begin{matrix}
R(-g+1)^{\oplus a_1} \\
\bigoplus \\
R(-g+2)^{\oplus a_{g-3}}
\end{matrix}
\to 
\begin{matrix}
R(-g+2)^{\oplus a_2} \\
\bigoplus \\
R(-g+3)^{\oplus a_{g-4}}
\end{matrix}
\to
\ldots \\
\ldots \to
\begin{matrix}
R(-4)^{\oplus a_{g-4}} \\
\bigoplus \\
R(-3)^{\oplus a_2}
\end{matrix}
\to
\begin{matrix}
R(-3)^{\oplus a_{g-3}} \\
\bigoplus \\
R(-2)^{\oplus a_{1}}
\end{matrix}
\to
R 
\to
0.
\end{multline}
The resolution is encoded in a \textit{Betti diagram}, which specifies 
the occurring exponents:  
\[
\begin{matrix}
1      &\cdot	 & \cdot   & \ldots &\cdot    &	\cdot   & \cdot\\
\cdot  & a_1     & a_2     & \ldots & a_{g-4} & a_{g-3}	& \cdot \\
\cdot  & a_{g-3} & a_{g-4} & \ldots & a_2     & a_1     & \cdot \\
\cdot  &\cdot    & \cdot   & \ldots &	\cdot &	\cdot   & 1. 
\end{matrix}
\]
Here the $\cdot$'s indicate zeros, and going either back a column or
up a row corresponds to twisting by $(-1)$.

Hilbert functions impose a further constraint.  Because Hilbert
functions are additive on exact sequences, and because the Hilbert
function of the $I_C$'s and $R(-k)^{\oplus j}$'s is known, it is
straightforward to derive a relation on the $a_i$'s:
\[
	a_i - a_{g-i-1} = i\binom{g-2}{i+1} - (g-i-1)\binom{g-2}{i-2}. 
\]
Thus, in the notation of \cite{Eis92}, the Betti diagram is
determined by the sequence
\begin{equation}\label{E:betti1}
	\mathbf{a} := \left(
	a_{\lfloor \tfrac{g}{2}\rfloor},\ldots,a_{g-3}\right).
\end{equation} 

The principal question motivating us is from \cite{Eis92}: 
\begin{question}[Eisenbud]
What sequences $\mathbf{a}$ can occur?
\end{question}
In \cite{Sch86}, all possible diagrams are determined through $g =
8$.  The primary goal of our research was to explore the possibilities  
for $g \ge 9$.

\section{Green's Conjecture}\label{S:conjecture} 
We give some definitions from \cite{Eis92}. 
\begin{definition}
The \textit{2-linear strand} of a resolution as in (\ref{E:exactSeq}) is
the subcomplex
\[
	 R(-g+2)^{\oplus a_{g-3}} \to  \ldots \to R(-3)^{\oplus a_2}
\to R(-2)^{\oplus a_1},
\]
which we specify by the sequence
\[
	(a_1,\ldots,a_{g-3}).  
\]
 
Because the resolution is minimal, if ever $a_i = 0$ then $a_j = 0$
for all $j \ge i$.  The \textit{2-linear projective dimension} $2LP$
of a minimal resolution is the length of its 2-linear strand, i.e., it is the 
number of nonzero entries.  
\end{definition}
The (nonhyperelliptic) curves for which $a_{g-3} \ne 0$ are exactly
the trigonal curves together with curves (of genus 6) which are
isomorphic to plane quintics.  
Green's conjecture may be viewed as an extension of this observation.
\begin{conjecture}[Green] For any (smooth) curve $C$, we have 
\[
	2LP = g-2-c,
\]
\end{conjecture}
\noindent
where $c$ is the Clifford index of $C$, i.e., it is the minimum over all
maps to $\PP^r$ of degree $d$ admitted by $C$ of the numbers 
\[
	d - 2r. 
\]

One direction of the conjecture is a theorem (for a more extensive explanation, 
see \cite{Eis92}):  if a curve admits a $d:1$ map to $\PP^1$ corresponding
to a line bundle $L$, then the $2\times 2$ minors of the matrix
corresponding to the multiplication map
\[
	V \otimes H^0(K_C \otimes L^{-1}) \to H^0(C,K_C) 
\]
cut out a rational normal scroll $S$ inside $\PP^{g-1} \cong
\PP(H^0(C,K_C))$ which contains $C$.  The 2-linear strand of the scroll
is a subcomplex of the 2-linear strand for $C$, and has $2LP \ge g-d$.
Thus the $2LP$ of the curve is $\ge g-d$, with $d$ minimal for $C$ --
i.e. with $d$ the so-called ``gonality'' of $C$.  Green's conjecture
(modulo the replacement of gonality with the refinement $c+2$) is that
the length is always exactly this.  Said otherwise, the number of
nonzero entries of the sequence $\mathbf{a}$ is conjectured to be
$\lfloor \tfrac{g+1}{2}\rfloor -c -1$.

\textit{Generic Green's conjecture} is a highly studied subconjecture,
asserting that there exists a curve of every genus $g \ge 3$ giving the
zero sequence $\mathbf{a} = (0,\ldots,0)$ (note the generic Clifford is 
index $c_{\text{gener}} = \lfloor \tfrac{g+1}{2}\rfloor -1$).  The
condition of having extra syzygies is closed, so if an example exists,
``almost all'' curves of that genus and characteristic must generate the
zero sequence.

\begin{remark}
Because of this, to verify generic Green's conjecture, it suffices to produce a
sequence $\mathbf{a} = (0,\ldots,0)$ coming from a singular curve, as
long at it is smoothable, i.e., it occurs in a flat family of smooth genus
$g$ curves. 
\end{remark}


\section{The REU}\label{S:REU}

The objective of our research was to explore the notions of free
resolutions and canonically embedded curves thru explicit computation,
and ultimately to illuminate Green's conjecture through experimental
evidence.  However, to even state the conjecture is complicated, so we
started more basically. 

The calculation of free resolutions is essentially elementary, being
based on Gr\"obner bases, which themselves are calculated via
polynomial long division. During the first days of the program, 
we spent writing code in the language of Macaulay2 \cite{GS} (henceforth abbreviated M2)  
to carry out polynomial long division. We
briefly studied the theory of Gr\"obner bases and syzygies for ideals using \cite[Chapter 2]{CLO97} and then extended
these notions to submodules of free modules \cite[Chapter 5]{CLO98}, 
writing algorithms designed to calculate generators of
kernels of maps of free modules.  We briefly discussed some
characteristics (e.g. minimality) of graded free modules \cite[Chapter 6]{CLO98} and then focused on the use of M2's
built-in functions. 

The next goal was to generate rings of canonical curves.  For this we
used as example \cite{Eis02}, which demonstrates the use of M2 in
calculating the canonical series of the normalization $C$ of a
singular plane curve $\Gamma$, 
given only a homogeneous
polynomial for $\Gamma$.  The result is a subring of the ring of the ambient
$\PP^2$, identifiable with the space of global regular differential
forms on $C$.  As this is the main idea behind the generation of our 
examples, we explain it more fully (see also \cite[Appendix A]{ACGH85}. 

Suppose a smooth, abstract curve $C$ of genus $g$ maps onto a
(possibly singular) plane curve $\Gamma$ of degree $d$.  Then the
entire space of global regular differential forms on $C$ can be
identified with a $g$-dimensional subspace of differential forms on
$\PP^2$, which are regular except possibly (simply) along $\Gamma$
itself.  The forms with simple poles along $\Gamma$ are identified
with homogeneous polynomials of degree $d-3$ in the coordinates of
$\PP^2$.  The subspace in question is determined by the types of
singularities of $\Gamma$. For instance, a simple node or cusp at a
point $P \in \PP^2$ requires limitation to the subspace of forms
vanishing at $P$.  Tacnodes and other more general singularities
require further conditions \cite[\#32, page 60]{ACGH85}.  If the
subspace under consideration has as basis a set of homogeneous
polynomials $\{F_0,\ldots,F_{g-1}\}$, then the subring
$k[F_1,\ldots,F_{g-1}] \subseteq k[X_0,\ldots,X_{g-1}]$ is identifiable with
the canonical ring of $C$.

The rest of our time in collaboration was spent generating canonical rings in order to
build a collection of diagrams. Details are given in the following
sections.

\section{Singularities and Adjunction} 
\label{S:singular} 
We begin by giving an example. 
\begin{example}
The following Betti diagram came from a genus 8 curve in characteristic 7:  
\begin{verbatim}
                        1  .  .  .  .  . .
                        . 15 36 33 12  1 .
                        .  1 12 33 36 15 .
                        .  .  .  .  .  . 1
\end{verbatim}
It has a two linear strand of length $g-3$, so should be trigonal, but
this diagram is not the unique diagram for trigonal genus 8 curves
(cf. \cite{Sch86}).  This would be a counterexample to Green's conjecture
(and \cite{Sch86}), except for the fact that the conjecture was made
for nonsingular curves.  The plane curve from which we produced the
diagram had singularities not defined over $k = \Z/7\Z$, and one of
its singularities over $k$ was of a type requiring stronger adjunction
conditions than were performed.
Thus the canonical model we produced by doing partial adjunction was singular. 
\end{example}

To signal such situations, we developed an M2 code to indicate when a
canonical curve was not smooth.  
The code can be obtained by contacting Mee Seong Im.  
It also contains supporting algorithms, e.g., resultants, as well as
subroutines for generating and analyzing lists of canonical rings. 
The following is a summary of the code:
a homogeneous polynomial $F$ of degree $d$ is entered (we will discuss more on how we
generated these later).  The partials $\partial F/\partial X_i$ are
calculated, and two cases can occur: 
\begin{enumerate}
\item every partial of $F$ passes through every point of $\PP^2_k$.
\item some partial does not pass through some point of $\PP^2_k$. 
\end{enumerate}
In either case, we can locate the $k$-singularities of $V(F)$ by
finding the points at which $F$ and all its partials vanish.  In the
first case, our code was not set up to say anything about the
possibility of there being other singularities of $V(F)$ not defined
over $k$.  We flag this by saying ``\verb%check1% = 4.''

In the second case, we proceed as follows:  a point $P$ over $k$ is
selected such that some derivative $\partial F/\partial X_i(P) \ne 0$.
Label the three partials $P_i$, $i = 0,1,2$, where $P_0$ is this
distinguished one.  Coordinates are changed to make $P = (1,0,0)$.
The three resultants \verb&R0&$i := \Res_{X_0}(P_0, P_i)$ ($i=1,2$) and
\verb&R0F&$ := \Res_{X_0}(P_0, F)$ are calculated, and all linear
factors defined over $k$, which they have in common, are put into a
list.  The resultant $\Res_{X_1}(\rem_{01}, \rem_{02})$ of the remainders
of \verb&R01& and \verb&R02& on dividing out all $k$-linear factors is
then calculated.  If it is identically zero and if $\Char \: k \nmid \deg
F$, then $V(F)$ has singularities not defined over $k$; we say
``\verb#check1# = 2.''  If $\Char \: k \mid \deg F$, we can conclude
nothing (this ambiguity could be resolved, e.g., by using
multipolynomial resultants on \verb&R0F& and the \verb&R0&$i$), and
say ``\verb#check1# = 1.''  If $\Res_{X_1}(\rem_{01}, \rem_{02})$ is not
identically zero, then all singularities of $V(F)$ lie on lines thru
$(1:0:0)$ which are defined over $k$, and we say ``\verb#check1# = 0.''

Next we set \verb#check2# = 0, and the (transformed) polynomial $F$
and the $P_i$ are restricted to the lines determined by $(1:0:0)$ and
the list of linear factors gathered above.  Common $k$-linear factors
are again gathered; these correspond to the subset of singular points
defined over $k$.  If the resultant of any pair of elements of
$\{\rem_{\overline F},\rem_{\overline P_0}, \rem_{\overline P_1}, \rem_{\overline P_2}\}$ (of
remainders of restrictions on dividing out all $k$-linear factors) is
not identically zero, then all singularities of $V(F)$ along this line
are defined over $k$, and we leave \verb#check2# = 0.  Otherwise, we
cannot conclude anything, and set \verb#check2# = 1.

Thus if ``\verb#check#'' is the minimum of the two checks, then we
have four cases:
\begin{enumerate}
\item \verb#check# = 0 implies all singularities of $V(F)$ are defined
over $k$.
\item \verb#check# = 1 implies there may be singularities of $V(F)$
not defined over $k$, but calculations are inconclusive.
\item \verb#check# = 2 implies there are definitely singularities of
$V(F)$ not defined over $k$.
\item \verb#check# = 4 implies the algorithm failed, either due to
lack of a point of $\PP^2_k$ not on some partial, or some other reason
(most commonly monomial overflow in the resultant calculations).
\end{enumerate}

Once the singularities have been located, they are classified
according to the following definition.
\begin{definition}
Let $P$ be a singularity over $k$ of the curve $C$, and let $f(x,y)$
be a local equation for $C$ such that $P$ corresponds to $(0,0)$.
Then $f$ decomposes into a sum of homogeneous parts $f = \sum_i
f_{(i)}$.  In this paper, a plane curve singularity will be called
\textit{admissible} if it is of one of the two following types:

\begin{enumerate}
\item a regular $m$-fold point for any $m \ge 2$.  These points are
such that $f_{(i)} \equiv 0$ for $0 \le i \le m-1$ and $f_m$
decomposes into $m$ distinct linear factors over the algebraic closure
$\overline k$ of $k$.  These are characterized by $f_{(m)}$ and its
derivative (with respect to either $x$ or $y$) having nonzero
resultant.  Such a point will be symbolized by $R_m$.

\item an $m$-th order node/cusp, i.e., $f_0 = f_1 = 0$, $f_2 = L^2$ for
some linear homogeneous $L$ defined over $k$, $L^2 \mid f_{(i)}$ for
$2 \le i \le m-1$, and $L \nmid f_m$.  Such a singularity will be
called a \textit{double point of index} $m$, and symbolized
by $D_m$.
\end{enumerate}
\end{definition}
These singularities were admitted for two reasons: 
\begin{enumerate}
\item we knew how to perform adjunction on them \cite[page 60]{ACGH85}:  an $R_n$ singularity $P$ requires restriction to homogeneous
forms $H$ of degree $\deg F-3$ which vanish to order $n-1$ at $P$ such that $H$ is in the $(n-1)$-st power of the ideal $I_P$ of $P$.
A $D_{2n}$ or $D_{2n+1}$ singularity $P$ with tangent line $L$
requires restriction to the forms in the ideal $(L) + {I_P}^n$.

\item both were necessary in generating our list of diagrams (see \S\ref{S:computations} for a further discussion).
\end{enumerate}

When a singularity was not admissible, we did adjunction as though it
were a regular singularity of order $\ord_P (f)$.  In that case,
adjunction produced only a partial normalization.  We flagged this
situation with a subscript $b$. For example, we would juxtapose the symbol
$2_{b1}$ with the sequence $\mathbf{a} = (5,0,0)$ if that sequence
were found in characteristic 2, but the $k$-singularities of $V(F)$
were not all admissible, and calculations were inconclusive as to
whether all singularities were defined over $k$.

\section{Connectedness}\label{S:connected} 

Next, we discuss a criterion for connectedness for our canonical
models.
\begin{observation}
If $P \in \PP^2_k$ is chosen so that the partial $F_{X_i}(P)$ of $F$ with
respect to some $i$ is nonzero, then $\Res_{X_0}(F,F_{X_i})$
is defined, and is identically zero if and only if $F$ is not reduced.
\end{observation}
Our code was written to abandon non-reduced curves.  

\begin{definition} 
Let $F$ be a reduced homogeneous polynomial of degree $f$ over some
field $k$.  The sub-intersection order $\subOrd_P(C)$ of a point $P$ on
$C := V(F)$ is the maximum over all intersection numbers
$I_P(V(H), V(K))$, where $H$ and $K$ are homogeneous of degrees $h$ and $k$, respectively, 
where $h+k = f$, and where the singularity $P \in V(HK)$ has
the same analytic type over $\overline k$ as that of $P \in V(F)$.
\end{definition}

\begin{example}
We have 
\[
	\subOrd_P(C) = 
\begin{cases}
	0 &\text{if $P$ of type $D_m$ and $m$ is odd}, \\
	\tfrac{m}{2} &\text{if $P$ of type $D_m$ and $m$ is even},\\
	\lfloor\tfrac{\ord_P(C)^2}{4}\rfloor &\text{if $P$ of type $R_m$}. 
\end{cases}
\]
\end{example}

If $F$ is reduced of degree $f$, then for $F$ to factor over $\overline k$
as $HK$, B\'ezout requires
\begin{equation}\label{E:necRed}
	hk \le \sum_{P \in \sing (V(F))} \subOrd_P(V(F)). 
\end{equation}

\begin{definition}
For $P \in V(F)$, let $\delta(P)$ be the amount by which the point $P$
drops the genus of $V(F)$. More precisely, $\delta(P)$ is the
dimension of the conditions on polynomials of degree $f-3$
imposed by the singularity in order to do adjunction.
\end{definition} 
\begin{example}
For $m \ge 2$, 
\[
	\delta(P) = 
\begin{cases}
	\lfloor \tfrac{m}{2}\rfloor &\text{if $P$ of type $D_m$},\\
	\binom{m}{2} &\text{if $P$ of type $R_m$}.
\end{cases}
\]
\end{example} 

Let $\delta := \sum_P \delta(P)$.  Then if $C$ has only admissible
singularities,
\begin{equation}\label{E:delvsSub}
	\sum_P \subOrd_P(C) \le \delta,
\end{equation}
with equality only if $C$ has no cusps or regular $R_m$-fold points
with $m > 2$.  We label 
\[
	g := \binom {f-1}{2} - \delta
\]
the geometric genus of the normalization of $V(F)$.  We thus have the following:  
\begin{proposition}\label{P:irred}
Suppose a plane curve $V(F)$ whose normalization has geometric genus
$g$ is specified by a homogeneous polynomial $F$ of degree $f$, and
$F$ factors as $F = HK$, where $h := \deg H \le k := \deg K$.  Then
\[
	h(f-h) \le \tfrac{1}{2}(f-1)(f-2)-g. 
\]
\end{proposition}

\begin{corollary}
If $V(F)$ is a plane curve with only admissible singularities, and if
$F = HK$ with $h \le k$, then $h$ can be no larger than the numbers
indicated in the following table:
\end{corollary}
\begin{table}[ht]
\begin{tabular}{|r||c|c|c|c|c|c|c|c|c|c|c|}
\hline
      &$g=4$& 5 & 6 & 7 & 8 & 9 & 10 & 11 & 12 & 13 & 14  \\ \hline
$f=5$ &   - & - & - &$*$&$*$&$*$& $*$& $*$& $*$& $*$& $*$ \\ \hline 
   6  &   1 & 1 & - & - & - & - &  - & $*$& $*$& $*$& $*$ \\ \hline 
   7  &   2 & 2 & 1 & 1 & 1 & 1 &  - &  - &  - &  - &  -  \\ \hline 
   8  &   4 & 4 & 3 & 2 & 2 & 2 &  1 &  1 &  1 &  1 &  1  \\ \hline 
   9  &   4 & 4 & 4 & 4 & 4 & 3 &  3 &  2 &  2 &  2 &  2  \\ \hline 
  10  &   5 & 5 & 5 & 5 & 5 & 5 &  5 &  5 &  4 &  3 &  3  \\ \hline 
\end{tabular}
\end{table}
Happily, our ``cases of interest'' coincide with the cases in the above
table having entries $\le 2$.  
That is, all computations fell under one of
the following sets of constrants:   
\begin{align}\label{E:coi}
f &\le 7 \quad \text{and} \quad g \ge 4,\\\notag
f &= 8 \quad \text{and} \quad g \ge 7,\\
f &= 9 \quad \text{and} \quad g \ge 11.\notag
\end{align}
In these cases, we have the following result: 
\begin{corollary}
In our cases of interest (\ref{E:coi}), a plane curve with only
admissible singularities can be reducible only if it contains a line
or conic defined over $k$.
\end{corollary}

\begin{proof}
As a singular conic must be a union of lines, we need only consider
the cases of a line and an irreducible (nonsingular) conic.  Our cases
of interest exclude the possibility of $V(F)$ being a union of lines.

Suppose first that $\deg H$ = 1, so that $V(H)$ must have intersection
number $f-1$ with $V(K)$.  $V(H)$ cannot be part of a singularity of
type $D_{k}$; its intersection number with $V(K)$ at an $R_n$
singularity is $n-1$.  Since $V(F)$ is not a union of lines, $n < f$,
so $V(H)$ must pass through two distinct points over $k$, and so is
definable over $k$.

Similarly, suppose $\deg H = 2$, so the intersection number of $V(H)$
with $V(K)$ is $2(f-2)$, and suppose $V(H)$ is a smooth conic.  We
will show that the singularities of $V(F)$ contained in $V(H)$ impose
five independent conditions over $k$ so that $V(H)$ is definable over
$k$.  Note first that any regular $n$-fold point over $k$ imposes one
condition over $k$, and causes $V(H)$ to have intersection number
$n-1$ with $V(K)$.  If $V(H)$ forms part of a singularity of type
$D_k$, then $k=4$, and two conditions (the point and a tangency
condition) are imposed. In this case $V(H)$ has intersection number
$2$ with $V(K)$.  For the following purposes we can think of a $D_4$
as being equivalent to two $R_2$'s.

So suppose there are four or fewer singularities of $V(F)$ through
which $V(H)$ is passing, say of type $R_{n_i}$, $i = 1,\ldots,4$.  Then
we must have
\begin{equation}\label{E:condn}
	\sum_{i=1}^4 (n_i-1) \ge 2(f-2), \qquad 
\text{i.e.}, \qquad \sum_{i=1}^4 n_i \ge 2f.  
\end{equation}
Let $\delta_0 := \sum_{i=1}^4 \binom{n_i}{2}$.  Then our cases of
interest (\ref{E:coi}) impose the added conditions
\begin{center}
\begin{tabular}{|c|c|c|c|c|c|c|}
\hline
$f =       $ & 4 & 5 & 6 &  7 &  8 &  9\\ \hline
$\delta_0 \le$ & 0 & 2 & 6 & 11 & 14 & 17 \\\hline  
\end{tabular}, 
\end{center}
which together with (\ref{E:condn}) deny any possibility.  
%
%
%
%
%
\end{proof}

\begin{corollary}\label{C:connected}
In our cases of interest (\ref{E:coi}), if $F$ has only admissible
singularities and is irreducible over $k$, then its normalization is
connected.  Even if $V(F)$ has singularities not defined over $k$ (but
its $k$-singularities are admissible), the partial normalization at
$k$-singularities is connected.
\end{corollary}


\begin{remark}
To check whether the normalization of a plane curve satisfying the
hypotheses of Proposition~\ref{P:irred} is connected, it thus suffices to check
for factorizability of $F$ over $k$.  M2's \verb@isPrime@ command does
this.
\end{remark}

\section{Computations, Expectations, and Results}\label{S:computations} 

Other than the data of the characteristic of $k$, two pieces of information were
necessary to generate our plane curves:
\begin{enumerate}
\item which points of $\PP^2(k)$ should we single out? 
\item what condition at each of these points should we impose?  
\end{enumerate} 
Once these decisions were made, we had M2 generate a ``random''
polynomial satisfying the prescribed constraints.~\footnote{It is
important to set a new random seed before each calculation (or M2 will
generate the same results repeatedly).  Our code uses the date/time
as random seed.}  This polynomial was a choice of element of the
linear system of all suitable curves.  Most of the time we expected to
get a ``generic'' such element by randomness, but (especially in
characteristics 2 and 3, when there was a dearth of choices for the
coefficients) sometimes running repeated assays with the same
constraints gave varying diagrams or (more often) allowed us to find a
nonsingular model for a diagram discovered using a singular curve.

Differences in diagrams most often were due to differences in the
number and type of singularities, so we concentrated on categorizing
these combinations.  To denote the possibilities, we use products of
$R_m$'s and $D_n$'s, e.g., the string $R_2^3R_4D_4$ implies the curve
had three regular nodes, a regular quadruple point, and a regular
tacnode (a double point of index 4 in the terminology of
\S\ref{S:singular}).  The degree of the plane curve is recoverable
from the genus of the canonical curve and the singularity combination.
In this notation we found that every possible diagram for curves of
genera $4 \le g \le 8$ for all prime characteristics
$2-11$ came from one of the configurations
\[
R_2^{i} \:\:\mbox{ for }\:\:0\leq i\leq 8,
\quad R_3, \quad R_2^4R_3, \quad
\text{or} \quad R_2R_4 
\]
(the generic diagram in genus 7 for characteristic 2 was an exception, which we will discuss more later). 
All possible diagrams in all possible
characteristics were classified in \cite{Sch86} for $g \le 8$.
Table~\ref{Ta:gle8} summarizes the sequences $\mathbf{a}$, together
with some of the combinations which worked to produce them.
\begin{table}[ht]
\begin{tabular}{|| c | r | l | l |}
\hline
$g$ & seq. $\mathbf{a}$ & Sings \\ \hline\hline 
4 & () & $R_2^2,D_4,R_2^6$\\ \hline\hline
5 & (0) & $R_2^5,R_2^3D_4,R_4D_4^5$\\ \hline
5 & (2) & $R_2$\\ \hline\hline
6  & (0) & $R_2^4$\\ \hline
6  & (3) & $\{\}$\\ \hline\hline
7  & (0,0) & $R_2^8,R_2^{11}R_3, R_2^2R_3^3D_6$\\ \hline
7  & *(1,0)  & $R_2D_4^2D_6 \: (\CHAR 2)$\\ \hline
\end{tabular}\quad
\begin{tabular}{|| c | r | l | l |}   
\hline 
$g$ & seq. $\mathbf{a}$ & Sings \\ \hline\hline 
7  & (3,0) & $R_2^8, R_2^{14}, R_2^5R_3$\\ \hline
7  & (9,0) & $R_2^3, R_2D_4, R_2^2R_4^2$\\ \hline
7  & (15,4) & $R_3$ \\ \hline\hline
8  & (0,0) & $R_2^7, R_2D_6^2, R_2^{11}D_3^2$\\ \hline
8  & (4,0) & $R_2^4R_3, R_2D_6^2, D_4, R_2^{13}$\\ \hline
8  & (14,0) & $R_2^2, D_4, R_2R_3^2$\\ \hline
8  & (24,5) & $R_2R_4, R_3R_5, R_2D_6^4$\\ \hline 
\end{tabular}
\caption{Genus $g \le 8$} \label{Ta:gle8} 
\end{table}
%
%
%
%
%

Emboldened, we began a categorization of configurations of regular
points for $9 \le g \le 11$ given $\deg F \le 8$.  We determined that
there are 34 allowable combinations:
\begin{description}
\item[Deg 6] $\{\}$, $R_2$, 

\item[Deg 7] $R_4$, $R_3^2$, $R_3R_2^3$, $R_2^6$, 
$R_3R_2^2$, $R_2^5$, $R_3R_2$, $R_2^4$, 

\item[Deg 8] $R_5R_2^2$, $R_4^2$, $R_4R_3^2$, $R_4R_3R_2^3$,
$R_4R_2^6$, $R_3^4$, $R_3^3R_2^3$, $R_3^2R_2^6$, $R_3R_2^9$,
$R_2^{12}$, $R_5R_2$, $R_4R_3R_2^2$, $R_4R_2^5$, $R_3^3R_2^2$,
$R_3^2R_2^5$, $R_3R_2^8$, $R_2^{11}$, $R_5$, $R_4R_3R_2$, $R_4R_2^4$,
$R_3^3R_2$, $R_3^2R_2^4$, $R_3R_2^7$, $R_2^{10}$.
\end{description}
Of these, 17 involve four or fewer points.  In these cases there are
at most three distinct configurations modulo $\PGL_3(k)$ (given the
constraints on the degree of $F$).  In fact we can always assign the
$n \le 4$ points to the first $n$ of $\{(1,0,0)$, $(0,1,0)$,
$(0,0,1)$, $(1,1,1)\}$ except in the presence of collinearity.  In
the following, we describe the exceptional configurations in analogy
to the following examples:
\[
\begin{matrix}
2 & 2\\
3 & 2
\end{matrix}
\]
denotes the placement of a triple point at $(1,0,0)$ and nodes at the
points $(0,1,0)$, $(0,0,1)$ and $(1,1,1)$.  
Next, 
\[
\begin{matrix}
2 &   &  \\
3 & 2 & 2
\end{matrix}
\]
denotes the same configuration up to the last node, which is not
placed at $(1,1,1)$, but along the line generated by $(1,0,0)$ and
$(0,1,0)$ (without loss of generality, this can be taken to be $(1,1,0)$ up to projective
transformation).  
The list of exceptional configurations (beyond the
default choice) is:
\[
\begin{matrix}
2 &&\\
3&2&2
\end{matrix}, \quad\:
\begin{matrix}
3 & & \\
2 & 2 & 2
\end{matrix}, \quad\:
\begin{matrix}
3 & & \\
4 & 2 & 2
\end{matrix}, \quad\:
\begin{matrix}
4 & & \\
3 & 2 & 2
\end{matrix}, \quad\:
\begin{matrix}
3 & & \\
3 & 3 & 2
\end{matrix},\quad\:
\begin{matrix}
3 & 2 & 2
\end{matrix},\quad\:
\begin{matrix}
2 &   &  \\
2 & 2 & 2.
\end{matrix} 
\]
Via this list, we generated examples for every possible configuration
up to $\PGL$ for allowable regular singularity combinations of four or
fewer points.

Four of the singularity combinations above involve exactly five
points:  $R_2^5$, $R_4R_3R_2^3$, $R_3^3R_2^2$ and $R_4R_2^4$.  To deal
with the possible configurations, note that in any case degree
restrictions deny any four to be collinear, so there are two kinds of
cases up to $\PGL$ equivalence:
\begin{enumerate}

\item four points including all non-ordinary nodes can be situated at
$\{(1,0,0)$, $(0,1,0)$, $(0,0,1)$, $(1,1,1)\}$, and the fifth can be
any other point of $\PP^2(k)$ (all possibilities must be considered), 

\item all points lie on two intersecting lines, with one of the points
at the point of intersection.  Diagramatically, the exceptional cases
are one of the following:
\[
\begin{matrix}
2&&\\
2&&\\
2&2&2
\end{matrix},\qquad\quad 
\begin{matrix}
2&&\\
4&&\\
2&3&2
\end{matrix},\qquad\quad 
\begin{matrix}
2&&\\
3&&\\
3&3&2
\end{matrix},\qquad \quad 
\begin{matrix}
2&&\\
2&&\\
2&4&2
\end{matrix},\qquad \quad 
\begin{matrix}
2&&\\
2&&\\
4&2&2
\end{matrix}.
\]
\end{enumerate}

There are four singularity combinations involving six points:
$R_2^6$, $R_3^3R_2^3$, $R_4R_2^5$, \text{and} $R_3^2R_2^4$.  Degree
considerations require at least three points (including all higher
order regular points) to be non-collinear, so without loss of generality, they can be taken
to be $\{(1,0,0),(0,1,0),(0,0,1)\}$.  Then, either there exists a
fourth point which is not collinear with any of these, or else we are
in the situation of one of the following exceptional configurations:
\[
\begin{matrix}
2&&\\
\boxed{2}&2&\\
2&2&2
\end{matrix},\qquad\quad 
\begin{matrix}
3&&\\
\boxed{2}&2&\\
3&2&3
\end{matrix},\qquad\quad 
\begin{matrix}
2&&\\
\boxed{2}&2&\\
3&2&3
\end{matrix},\qquad \quad 
\begin{matrix}
2&&\\
\boxed{2}&2&\\
4&2&2
\end{matrix},\qquad \quad 
\begin{matrix}
2&&\\
\boxed{2}&&\\
2&&\\
2&2&4
\end{matrix}.
\]
Here the boxed number is allowed to vary along the indicated line.
Without loss of generality, the non-varying points can be taken to be $(1,0,0)$, $(1,1,0)$,
$(0,1,0)$, $(1,0,1)$, $(0,1,1)$, and $(0,0,1)$.

\begin{remark}
Any diagram coming from an exceptional configuration also came from a
``generic'' one, when one existed.  Thus the exceptional
configurations contributed nothing new.  
\end{remark}

The remaining singularity combinations are of one of the forms:
$R_4R_2^6$, $R_3R_2^{i}$ ($i=5,6$), $R_3R_2^{j}$ ($j=7,8,9$), or $R_2^{k}$ ($k=10,11,12$).  During
continuing work after our summer research (in lieu of an exhaustion of all
equivalence classes of configurations for these combinations), we
generated lists of randomly chosen points and automated M2 to
search on its own.  We were able to examine on the order of 10,000
curves in this way.  The results of these calculations are given as
the sequences in Tables~\ref{Ta:g9}--\ref{Ta:g12} which are not preceded by
asterisks.  

We found it necessary in characteristic 2, genus 7 to use non-regular
singularities to get the generic diagram, which is different than the
generic diagram for characteristic $p \ne 2$.  The generic diagrams
for the other characteristics in this genus were all found using at least
8 points at which to specify regular singularities.
Unfortunately the number of points of $\PP^2_{k}$ for $k=\Z/2\Z$ is
only 7 (we could circumvent this issue by allowing
coefficients in extensions of $k$; we were not able to get our code to
accommodate these, but this seems like good material for a sequel). 

Exotic double points are particularly useful in such situations:  a
singularity with local equation analytically equivalent to $y^2 = x^N$
reduces the genus by $\lfloor \tfrac{N}{2} \rfloor$, but affects
gonality in the same way as a regular double point.  
Describing an algorithm to categorize configurations and types
of singularities including exotic double points might thus make an
interesting project. Due to time constrants, we however did little in
this direction.


Tables \ref{Ta:g9}--\ref{Ta:g12} list our results (the meanings of the
subscripts on the characteristics are discussed in
\S\ref{S:singular}).  All entries in these tables came from curves
which were irreducible over $k$ according to M2's \verb#isPrime#
command. By Proposition \ref{C:connected}, the curves in the tables
indicated by characteristics without the subscript $b$ are connected.
Sequences preceded by asterisks were not obtained using only regular
singularities.  The existence of sequences with no nonzero entries
verifies Generic Green for the corresponding genus and characteristics
(see \S\ref{S:genericGreen} for more discussion).

\begin{table}[ht]
\begin{tabular}{||r| c | l |}
\hline
sequence $\mathbf{a}$ & char & sings\\ \hline\hline 
(0, 0, 0)  & $2_1,5,7,11$ & $R_2^{12}, D_6^4, R_2R_3, D_4D_6^2, R_2^3D_6^3, D_4^3D_6^2$\\ \hline
(4, 0, 0)  & $2,5,7,11$   & $R_2^{12}, R_2R_3D_4^4,R_2R_3D_4D_6^2$\\ \hline
*(6, 0, 0)  & $3$		  & $D_6^4, R_2^{10}D_3^2$\\ \hline  
(8, 0, 0)  & $2,3,5,7,11$ & $R_2^{12}, R_2^6R_3^2, D_6^4, R_2^2R_3^2D_4^2, D_3D_4^4D_6,R_3^2D_6D_7$\\ \hline
*(10, 0, 0) & $3$ 			& $R_2^6D_4^3, R_2^4R_3^2D_5$\\ \hline 
(12, 0, 0)  & $2,3,5,7,11$ & $R_2^9R_3, R_2^6R_3^2, R_2R_3D_4^4, R_3D_6^3, R_2R_3D_4D_6D_7$\\ \hline 
(24, 0, 0)  & $2,3,5,7,11$ 	& $R_2, R_2^6, D_6^2$ \\ \hline
(24, 5, 0)  & $2,3,5,7,11$ 	& $R_2^6R_4, R_2^{12}, R_2^2D_4^2D_6^2, R_4D_6^2$\\ \hline 
*(28, 5, 0) & $3_2, 3_b, 7_b$ & $R_2^9D_3^3, R_2^2D_4^2D_6^2, R_2^3D_6^3$\\ \hline 
(44, 5, 0)  & $2,3,5,7,11$ 	& $R_3^4, R_2^3R_3R_4, R_2R_3D_4, R_3D_6$\\ \hline
(64, 20, 0) & $2,3,5,7,11$ 	& $R_2, R_4^2, D_8^3$\\ \hline
(84, 35, 6) & $2,3,5,7,11$ 	& $R_4, R_2^2R_5$\\ \hline
\end{tabular}\qquad
\caption{Genus 9}\label{Ta:g9}
\end{table}

\begin{table}[ht]
\begin{tabular}{|| r | c | l |}
\hline
sequence $\mathbf{a}$ & char & sings\\ \hline\hline
(0, 0, 0)	& $2,5,7,11$ 	& $R_2^{11}, R_2^9D_3^2, R_2^2D_6^3, D_4^2D_6D_8$\\ \hline
*(1, 0, 0)	& $3$ 		& $R_2D_4^5, R_2D_3^{10},R_2^2D_3D_6^2$ \\ \hline
(5, 0, 0)	& $2,3,5,7,11$	& $R_2^8R_3, R_3D_4^4$ \\ \hline
*(6, 0, 0)	& $3_{b2}$ 		& $R_2^3R_3D_4D_6$ \\ \hline
(10, 0, 0)	& $2,3,5,7,11$	& $R_2^{11},R_2^5R_3^2, R_2^4R_3^2D_4, R_3^2D_5D_7$ \\ \hline
*(12, 0, 0)	& $3_b, 7_{b2}$ 	& $R_2^{9}D_4, R_2^3D_3D_4^2D_6, R_2^3R_3^2D_3, R_2^2D_4^3D_6$ \\ \hline 
*(20, 0, 0)  & $2_1,3,5,7,11$ & $R_2D_4^2D_6^2, R_2^4D_3D_4^3, R_2^8D_3^3, R_2^{10}D_3, R_2D_3D_4^2D_6$\\ \hline 
(35, 0, 0)   & $2,3,5,7,11$ 	& $R_2^5, R_2^2R_3^3, D_4D_6, R_2^4R_4$ \\ \hline
(35, 6, 0)   & $2_1,2_b,3,5,7,11$ & $R_2^5R_4, R_2^3R_4D_4$ \\ \hline 
(70, 6, 0)   & $2,3,5,7,11$ 	& $R_2^2R_3, R_2^2R_3R_4$ \\ \hline
(105, 27, 0) & $2_1,3_1,5,7,11$ & $\{\}$ \\ \hline
(140, 48, 7) & $2,3,5,7,11$ 	& $R_2R_5, R_3R_6$ \\ \hline
\end{tabular}
\caption{Genus 10}\label{Ta:g10}
\end{table}

\begin{table}[ht]
\begin{tabular}{|| r | c | l |}\hline
sequence $\mathbf{a}$ & char & sings \\ \hline\hline
(50, 0, 0, 0)     & $5,7,11$ 	& $R_2^{10}, R_2^7D_3^3, R_2^8D_3^2, R_2^4D_4^3$ \\ \hline
*(56, 0, 0, 0)    & $3$ 	& $R_2^5D_3^5$ \\ \hline
(60, 0, 0, 0)     & $2,5,7,11$& $R_2^{10}, R_2^4D_4^3, R_2^8D_3^2, D_3D_6^2D_7$ \\ \hline
*(64, 0, 0, 0)    & $2,3$ 	& $R_2^9D_3, R_2^8D_3^2, D_3D_6^3, R_2D_6^3$ \\ \hline
(65, 0, 0, 0)     & $5,7$ 	& $R_2^{10}, R_26D_3^4$ \\ \hline
*(68, 0, 0, 0)    & $2,3$ 	& $R_2^5D_3^5, R_2^6D_3^4, R_2^2D_4D_6^2$ \\ \hline
*(72, 0, 0, 0)	& $2_{b1}$	& $R_2^2D_4D_6^2$ \\ \hline
*(74, 0, 0, 0)    & $3_2$ 	& $R_2^9D_2$ \\ \hline
*(100, 0, 0, 0)   & $2,3,5,7,11$ & $R_2^7D_3^3, R_2^5D_3^5, R_2^4D_4^4, R_2^9D_3, R_2^6D_3^4, D_7^2D_8$ \\ \hline
(75, 6, 0, 0)     & $5,7,11$ 	   & $R_2^7R_3$ \\ \hline
(76, 6, 0, 0)     & $3$	 	   & $R_2^7R_3, R_2^3R_3D_4^2$ \\ \hline
*(78, 6, 0, 0)	& $2$		   & $R_2R_3D_6D_7, R_2R_3D_6^2$ \\ \hline
*(80, 6, 0, 0)    & $2,3,5,7,11$ & $R_2^6R_3D_3, R_2^3R_3D_3^4, R_3D_6D_8$ \\ \hline
(140, 12, 0, 0)   & $2,3,5,7,11$ & $R_2^4R_3^2$ \\ \hline
(210, 48, 0, 0)   & $2,3,5,7,11$ & $R_2^4, R_2R_3^3, D_4^2, R_2D_6$ \\ \hline
(210, 48, 7, 0)   & $2_1,2_b,3,5,7,11$ & $R_2^4R_4, R_2R_4D_3D_4$ \\ \hline 
(280, 104, 7, 0)  & $2,3,5,7,11$ & $R_2R_3R_4$ \\ \hline
(420, 216, 63, 8) & $2,3,5,7,11$ & $R_5, R_2R_6D_3$ \\ \hline
\end{tabular}
\caption{Genus 11}\label{Ta:g11}
\end{table}

\begin{table}[ht]
\begin{tabular}{|| c | r | c | l |}
\hline
genus & sequence $\mathbf{a}$ & char & sings \\ \hline\hline
12 & *(69, 0, 0, 0)   & $2,3,7$ 	       & $R_2^3D_6^2, D_6^3$ \\ \hline 
12 & *(75, 0, 0, 0)   & $2_1,3,5$	       & $R_2^3D_6^2, D_6^3$ \\ \hline
12 & *(105, 7, 0, 0)  & $2_1,2_b,3,5,7$    & $R_3D_6^2, R_2^4R_3D_4, R_2R_3D_4D_6$ \\ \hline
12 & *(216, 14, 0, 0) & $2_1,2_b,3,5,7,11$ & $R_3^2D_7$ \\ \hline
12 & (342, 63, 0, 0)  & $2,3,5,7,11$       & $R_2^3, R_3^3, R_2D_4, D_6$ \\ \hline 
12 & (342, 63, 8, 0)  & $2_1,3,5,7,11$     & $R_2^3R_4$ \\ \hline 
12 & (468, 147, 8, 0) & $2,3,5,7,11$       & $R_3,R_3R_4$ \\ \hline
12 & (720, 315, 80, 9)& $2,3_1,5,7,11$     & $R_2R_6$ \\ \hline\hline 
13 & (972, 315, 16, 0, 0) 	& $3_b, 5_b$     & $R_2R_3^2D_3, R_2^2R_3^2$ \\ \hline  
13 & (1224, 525, 80, 0, 0) 	& $2,3,5,7_1,11$ & $R_2^2, D_4$ \\ \hline
13 & (1224, 525, 80, 9, 0)    & $2_1,3,5,7,11$ & $R_2^2R_4$ \\ \hline
13 & (1980, 1155, 440, 99, 10)& $2,3_1,5,7,11$ & $R_6$ \\ \hline\hline 
14 & (1617, 440, 18, 0, 0)  & $2_{b1}$ 		    & $R_2R_3^2$\\ \hline 
14 & (2079, 770, 99, 0, 0)  & $2,3,5,7_{1},11$	    & $R_2, D_3$ \\ \hline
14 & (2079, 770, 99, 10, 0) & $2_1,3,5,7$ 	    & $R_2R_4$ \\ \hline 
14 & (3465,1760,594,120,11) & $2_4,3_4,5_4,7_4,11_4$& $R_2R_7$ \\ \hline
\end{tabular}
\caption{Genus $g \ge 12$}\label{Ta:g12}
\end{table}

\section{Generic Green}\label{S:genericGreen} 
A side goal was to verify generic Green's conjecture for our
acceptable genera and characteristics.  The conjecture is false in
prime characteristic (in particular in characteristic 2, genus 7), and
most of the literature (with the notable exception of \cite{Sch86})
is concerned with the genus zero case.  

\begin{remark}
A computational strategy for characteristic 0 involves the use of
$g$-cuspidal rational normal curves in $\PP^g$.  It fails completely
for characteristic 2 (the resulting rings are evidently not even
Gorenstein) and fails to give the generic diagrams for characteristics
3 and 5 for the genera of our study (Table~\ref{Ta:cusp}; for
background on this strategy, see \cite[Project 7,
pages 379--81]{Eis95} or \cite[pages 61--70]{Eis92}). 
\end{remark}

We were able to get generic diagrams for characteristics 2, 5, 7 and
11 (and other small $p\ne 3$) in genera 9 and 10, as well as the known
diagrams for $g \le 8$.  Interestingly, the closest to the generic
diagram for $\CHAR 3$ in genus 10 was the sequence $(1,0,0)$, found in
$\CHAR 3$ only, reminiscent of the case for $\CHAR 2$ in genus 7.
Something similar happens in genus 9 (see Table~\ref{Ta:g9}).  We
naively posit the following: 
\begin{conjecture}
Generic Green fails for $\CHAR 3$ in genera 9 and 10.  
\end{conjecture}

There was no hope of verifying generic Green for genera $g \ge 11$
given our calculational constraints. For example, any curve coming
from a degree $8$ plane curve with one ordinary node has gonality $\le
6$, and the general curve of genus $11$ has gonality $7$.  We were not
able to get M2 to handle the resolution of normalizations of curves of
degree $f \ge 9$ with many singularities of small order.

\begin{table}[ht]
\begin{tabular}{|| r || r | r | r | r |}
\hline
$g$ & char = 3 & 5 & 7 \\ \hline\hline
9 & (84, 35, 6) & (4, 0, 0) & (0, 0, 0) \\ \hline
10 & (140, 48, 7) & (5, 0, 0) & (0, 0, 0) \\ \hline
11 & (420, 216, 63, 8) & (35, 6, 0, 0) & (0, 0, 0, 0) \\ \hline
12 & (720, 315, 80, 9) & (48, 7, 0, 0) & (0, 0, 0, 0) \\ \hline
13 & (1980, 1155, 440, 99, 10) & (274, 63, 8, 0, 0) &(6, 0, 0, 0, 0)
\\ \hline
14 &(3465, 1760, 594, 120, 11) &(315, 80, 9, 0, 0) &(7, 0, 0, 0, 0)
\\ \hline
\end{tabular}
\caption{Sequences $\mathbf{a}$ from $g$-cuspidal rational normal
curves}\label{Ta:cusp}
\end{table}

\section{Degenerate Examples}\label{S:reducible}
A natural question is whether every sequence coming from a Gorenstein,
two-dimensional ring $S$ of degree $2g-2$ comes from a normal
canonical curve, i.e., whether nonsingularity or connectedness impose
significant conditions.  As indicated by the Example of
\S\ref{S:singular}, some sequences do not.

A few 2-linear strands coming from reducible curves are exhibited in
Table~\ref{Ta:red}; each came from a 2-dimensional Gorenstein ring of
degree $2g-2$ for the indicated genus $g$.  

\begin{table}[ht]
\begin{tabular}{|| c | r | l |}
\hline
genus & 2-linear strand & char \\ \hline 
6  & (6, 6, 1)              		& $3_b$\\ \hline 
7  & (10, 18, 11, 2)        		& $3_{b2}$\\ \hline
8  & (15, 35, 22, 1, 0)     		& $3_b$  \\ \hline
9  & (21, 64, 71, 24, 1, 0) 		& $3_2, 5$ \\ \hline 
9  & (21, 64, 75, 32, 5, 0)		& $2_4$ \\ \hline
9  & (21, 65, 76, 36, 6, 1) 		& $7_b, 11_2$ \\ \hline 
10 & (28, 105, 162, 90, 6, 0, 0)	& $5_{b2}$ \\ \hline
10 & (28, 105, 162, 101, 17, 0, 0)	& $2_4$ \\ \hline
10 & (28, 105, 163, 104, 20, 1, 0) 	& $7_2$ \\ \hline 
\end{tabular}
\caption{2-linear strands from reducible canonical curves}\label{Ta:red}

\end{table}

\section{Further questions and projects}\label{S:further}

In this section, we provide a list of possible future research projects. 

\begin{enumerate}



\item Can a Macaulay2 code be written to make effective use of coefficients in
finite fields $\mathbb{F}_{p^n}$?  

\item To what extent can quadratic transformations be used to improve
bad singularities?  Is it computationally feasible to use them to
improve our adjunction algorithm?

\item Can a Macaulay2 code be written to do adjunction on an arbitrary plane
curve, e.g., using the \verb&integralClosure& and \verb&ICmap&
operations?  Using these functions, can hyperelliptic curves be
included in the analysis?

\item Can a meaningful analysis be done for curves embedded in spaces
other than $\PP^2$?  We, for instance, considered the intersection of a
general quadric and general quartic in $\PP^3$, and the intersection of
two general cubic surfaces in $\PP^3$.  These gave the sequences
$(64,20,0)$ (genus 9) and $(20,0,0)$ (genus 10), respectively.  We did
not try looking at (normalizations of) singular curves produced in a
way similar to this.

\item Can the non-generic diagrams calculated for the cuspidal
rational normal curve example in Table \ref{Ta:cusp} be realized by
genus $g$ canonical curves?  If so, what is the connection? For example, in
characteristic 3, the sequences are those of trigonal curves.  More
generally, what diagrams can be realized as degenerations of others?

\item We generated our curves $V(F)$ by selecting an element of the
linear system of all curves of degree $d$ satisfying some imposed
constraints.  How are the diagrams of all elements of such a linear
system related? For example, what happens when we perturb coefficients?  Is
there a canonical choice always giving the ``generic" element (for
instance the sum of all generators of the system as calculated by M2)?

\item Is there an effectively computable way to determine the gonality 
(or Clifford index) of a plane curve given by a degree $d$ homogeneous 
polynomial?  Carefully check whether Green's conjecture is verified or 
denied by the examples herein (or others).  

\item For a general curve $C$ of genus $g$, what is the minimum degree 
$d$ of a plane curve $\Gamma$ onto which $C$ surjects?  How many
nodes or other singularities are required?  

\item What effect does assigning nonsingular points to plane curves
have?  For example, can new diagrams be obtained by requiring curves only to
pass through certain configurations of points (in addition to having
prescribed singularities at certain points)?

\item Do a study on reducible or non-reduced curves.  Is there a
criterion for determining whether a given Betti diagram comes from a reducible curve?

\item In each of the genera $g \in \{9,\ldots,14\}$, there were pairs of
  sequences of the form 
\begin{align*}
  (\ldots, &(g-5)(g-3), \:\quad 0 \:\quad, 0 ) \\
  (\ldots, &(g-5)(g-3), (g-4), 0 ).
\end{align*} 
The first was obtainable with $R_2^k$ and the second with $R_2^kR_4$,
where $k = 15-g$.  What are geometric descriptions of the corresponding curves,
and what is their relationship?  What other patterns can be found in
the data?  

\item More generally, what does the locus with given sequence
  $\mathbf{a}$ look like in the moduli space of curves?
\end{enumerate}

 
 

\bibliography{litlist-green} \label{references}
\bibliographystyle{plain}

\end{document}